\documentclass[10pt]{amsart}
\usepackage{amsfonts,amssymb,amscd,amsmath,amsrefs,latexsym,verbatim,calc,
quotmark,enumitem}

\newcommand{\I}{\mathrm{I}}

\def\hom{\operatorname{Hom}}
\def\Aut{\operatorname{Aut}}
\def\max{\operatorname{Max}}
\def\dim{\operatorname{dim}}
\def\rank{\operatorname{rank}}
\def\E{{\operatorname{E}}}            
\def\GL{{\operatorname{GL}}}          
\def\K{{\operatorname{K}}}            

\def\og{{\operatorname{O}}}          
\def\SO{{\operatorname{SO}}}          
\def\EO{{\operatorname{EO}}}          
\def\Sp{{\operatorname{Sp}}}
\def\ESp{{\operatorname{ESp}}}

\def\KH{{\operatorname{KH}}}

\newtheorem{theorem}{Theorem}[section]
\newtheorem{lemma}[theorem]{Lemma}
\newtheorem{corollary}[theorem]{Corollary}
\newtheorem{proposition}[theorem]{Proposition}

\theoremstyle{definition}

\newtheorem*{acknowledgements}{Acknowledgements}

\title[Normality of DSER elementary orthogonal group ]%
 {Normality of DSER elementary orthogonal group} %

\author{A.A. Ambily \and Ravi A. Rao}

\address{\newline Department of Mathematics \hfill School of Mathematics
\newline Cochin University of Science and Technology\hfill Tata Institute of Fundamental Research
\newline Cochin 682022, Kerala\hfill 1, Dr. Homi Bhabha Road, Mumbai 400 005
\newline India \hfill India}
\email{{ambily@cusat.ac.in, aaambily@gmail.com}}
\email{{ravi@math.tifr.res.in}}

\thanks{ 2010 Mathematics Subject Classification \  19G99 (primary), 19B99, 20G35, 20H25 (secondary)}

\begin{document}
\maketitle
\thispagestyle{empty}
\begin{abstract}
Let $(Q, q)$ be a quadratic space over a commutative ring $R$ in which $2$ is invertible, and consider the Dickson--Siegel--Eichler--Roy's subgroup $\EO_{R}(Q, H(R)^{m})$ of the orthogonal group $\og_R(Q \perp H(R)^m)$, with $\rank Q= n \geq 1$ and $m\geq 2$. We show that $\EO_{R}(Q, H(R)^{m})$ is a  normal subgroup of $\og_R(Q \perp H(R)^m)$, for all  $m\geq 2$. We also prove that the DSER group $\EO_{R}(Q, H(P))$ is a normal subgroup of  $\og_{R}(Q \perp H(P))$, where $Q$ and $H(P)$ are quadratic spaces over a commutative ring $R$, with $\rank(Q) \ge 1$ and $\rank(P) \ge 2$.\end{abstract}

\section{Introduction}

J-P. Serre's problem on projective modules (See \cite{MR0068874}) played a vital role in the development of $K$-theory. It stimulated the investigation of $K_{1}$, higher $K$-groups, the stable and unstable behaviour of $K$-groups. The study of $K_{1}$ began with Bass's seminal work in 1964 (See \cite{MR0174604}).  In \cite{MR0207856}, P.M. Cohn proved that $\E_{2}(R)$ is not normal in $\GL_{2}(R)$ for a commutative ring $R$. In 1977, A.A. Suslin proved the normality of the elementary linear group $\E_{n}(R)$ in the general linear group $\GL_{n}{(R)}$ for $n\ge 3$, over a commutative ring $R$ (See \cite{MR0447424}). Analogues of Suslin's normality theorem were proved for the elementary orthogonal group $\EO_{2n}(R) \subset \og_{2n}(R)$ for $n\ge 3$ by A.A. Suslin--V.I. Kope{\u\i}ko in \cite{MR0469914} and the failure of the result in the case $n=2$ was proved by V.I. Kope{\u\i}ko in \cite{MR717580}. Similar results for the elementary symplectic group $\ESp_{2n}(R) \subseteq \Sp_{2n}(R)$ for $n\ge 3$ was proved by Kope{\u\i}ko in \cite{MR497932}. The normality results for the elementary Chevalley groups over local rings was proved by E. Abe in \cite{MR0258837} and for elementary Chevalley--Demazure group schemes over commutative rings was studied by G. Taddei in \cite{MR862660}.  In \cite{MR1958377}, R. Hazrat--N. Vavilov gave a simpler proof of normality of elementary Chevalley groups. The case of twisted elementary Chevalley groups was proved by K. Suzuki in \cite{MR1339654} and A. Bak--N. Vavilov in \cite{MR1329456}.

\vspace{1.5mm}

For classical-like groups such as general quadratic groups, general Hermitian groups, odd unitary groups, isotropic reductive groups, the normal subgroup structure was studied by A. Bak--N. Vavilov in \cite{MR1329456}, G. Tang in \cite{MR1609905}, V. Petrov in \cite{MR2033642}, V.Petrov--A. Stavrova in \cite{MR2473747} and W. Yu in \cite{MR3003312} respectively. Recently, in \cite{MR3492199}, W. Yu--G. Tang studied the nipotency of  odd unitary $K_{1}$-functor and in \cite{RB1}, R. Basu  showed that the unstable $K_{1}$-groups of general Hermitian groups over module finite rings are nilpotent-by-abelian.

 \vspace{1.5mm}

 We study the normality of the elementary Dickson--Siegel--Eichler--Roy's orthogonal group which was considered by A. Roy in \cite{MR0231844}. We describe this briefly below.
 
 \vspace{1.5mm}

 Let $R$ be a commutative ring and let $(Q,q)$ be a quadratic $R$-space with associated bilinear form $B_q$ and $P$ be a finitely generated projective $R$-module. The module $P\oplus P^*$ has a natural quadratic form given by $p((x,f)) = f(x)$ for $x\in P$, $f\in P^*$. The corresponding bilinear form $B_p$ is given by $B_p((x_1,f_1),(x_2,f_2)) = f_1(x_2)+f_2(x_1)$ for $x_1,x_2 \in P$ and $f_1,f_2 \in P^*$. The quadratic space $(P \oplus P^*, p)$, denoted by
$H(P)$, is called the {\em hyperbolic space} of $P$. Given any homomorphism $\alpha : Q \rightarrow P$, define $\alpha^*: P^* \rightarrow Q$ by the formula $\alpha^* = d_{B_q}^{-1} \circ \alpha^t$, 
where $\alpha^t$ denotes the dual map $P^* \rightarrow Q^*$. If  $\beta : Q\rightarrow P^*$, then define $\beta^*: P \rightarrow Q$ by precomposing $d_{B_q}^{-1}\circ \beta^t$ with $\varepsilon: P\rightarrow P^{**}$, where $\beta^t$ denotes the dual map $P^{**} \rightarrow Q^*$.
The $A$-linear map $\alpha^*$ is characterized by the relation \[(f\circ\alpha)(z) = B_q\left(\alpha^*(f),z \right)  \textnormal{ for } \,f\in P^*,\, z\in Q.\]

Let $\og_{R}(Q)$ denote the orthogonal group of the quadratic module $(Q,q)$. That is, 
\[\og_{R}(Q) = \{\alpha \in \Aut_{R}(Q)\  | \  q(\alpha(z)) =q(z) \mbox{ for all } z\in Q\}\]
In \cite{MR0231844}, A. Roy defined the ``elementary'' transformations
$E_{\alpha}$ and $E_{\beta}^*$ of $Q\!\perp\!H(P)$ given by
\[\begin{array}{lllll}
\medskip 
E_{\alpha}(z) &= z+\alpha(z) & & E_{\beta}^*(z) &=z+\beta(z)\\ 
\medskip
E_{\alpha}(x) &= x & & E_{\beta}^*(x) &= -\beta^*(x)+ x-\frac{1} {2}\beta\beta^*(x)\\ 
\medskip
E_{\alpha}(f) &= -\alpha^*(f)-\frac{1}{2}\alpha\alpha^*(f)+f & & E_{\beta}^*(f) &= f
\end{array}\] 
for $z \in Q, x \in P$ and $f\in P^*$.  

 \vspace{1.5mm}

The Dickson--Siegel--Eichler--Roy's subgroup $\EO_{R}(Q\perp H(P))$ of the orthogonal group $\og_{R}(Q\perp H(R)^{m})$ is the subgroup generated by the elementary generators $E_{\alpha}$, $E_{\beta}^{*}$, where $\alpha \in \hom(Q,P)$ and $\beta \in \hom(Q,P^{*})$. We shall refer to this group by DSER group. For more historical details on this group we refer the reader to the introduction of \cite{aarr}. 

 \vspace{1.5mm}

In \cite{MR3525594}, B. Calm\`es--J. Fasel have described an elementary subgroup $\EO_{2n+1}(R)$ of the special orthogonal group $\SO_{2n+1}(R)$ generated by a set of 5 generators. By comparing the generators with the generators of DSER, we can identify two of the elementary generators are of the form $E_{\alpha_{ij}}$ and $E_{\beta_{ij}}^{*}$, and other three generators are commutators of these type of generators.
 \vspace{1.5mm}

The following normality results were proved by the first author in \cite{aa1}.
  \vspace{1.5mm}

 \begin{enumerate}[label = (\roman*)]
 \item {\it $ \og_R(Q\!\perp\!H(R)^{m-1})$ normalizes $ \EO_R(Q\!\perp\!H(R)^{m}).$ In particular$,$ $\EO_R$ is a normal subgroup of $\og_R$.}
\item {\it If $m \geq \dim \max(R) + 2,$ then $ \og_R(Q\!\perp\!H(R)^{m})$ normalizes $
\EO_R(Q\!\perp\!H(R)^{m})$.}
\item {\it If $m > l,$ then $\og_R(Q\!\perp\!H(R)^{m})$ normalizes 
$\EO_R(Q\!\perp\!H(R)^{m})$  provided $R$ satisfies the stable range condition $0$-$SR_l$.}
\end{enumerate}
 \vspace{1.5mm}

All aforementioned results about normality of this elementary 
orthogonal group in the literature have assumed stable rank 
conditions. In this paper, we prove that the normality of the DSER subgroup in most generality {\it without these restrictions on the hyperbolic rank}, except for the case of $\rank(Q) =2$ and $m=1$. We prove the following theorems.
\begin{theorem}\label{main1}
  $\EO_{R}(Q, H(R)^m)$ is normal in  $\og_{R}(Q\perp H(R)^m)$, where $Q$ and $H(R)^m$ are quadratic spaces over a commutative ring $R$ with $\rank(Q) \ge 1$ and $m \ge 2$.
\end{theorem}
\begin{theorem}
The DSER group $\EO_{R}(Q, H(P))$ is a normal subgroup of  $\og_{R}(Q \perp H(P))$, where $Q$ and $H(P)$ are quadratic spaces over a commutative ring $R$, with $\rank(Q) \ge 1$ and $\rank(P) \ge 2$.
\end{theorem}
In order to prove these theorems, we prove an analogue of Quillen's Local-Global principle in the extended module case for the DSER group. In addition to these results, we see that in the free module case, the elementary DSER group with $q$ to be the standard hyperbolic form coincides with the usual elementary orthogonal group $\EO_{n+2m}(R)$. We also use a decomposition result of R.A. Rao in \cite{MR727375}.
\newpage

\noindent {\bf Notations:}
\begin{itemize}
\item $\GL_n(R)$ will denote the general linear group. Let $\alpha \in \GL_n(R)$ and $\beta \in \GL_m(R)$, then by $\alpha\perp \beta$ we denote the matrix 
$\begin{pmatrix}
\alpha &0\\
  0&\beta                                                                                                      \end{pmatrix}$
and by $\alpha\top \beta$ we denote the matrix $\begin{pmatrix}
                        0&\alpha\\
                        \beta&0                                                                                                     \end{pmatrix}$.
\item $\widetilde{\psi_1} = \begin{pmatrix} 0 & 1\\1 & 0\end{pmatrix}$, $ \widetilde{\psi_r} = \widetilde{\psi_1} \perp \widetilde{\psi_{r-1}}$.
\item The matrix of the bilinear form corresponding to the quadratic form $q$ is denoted by $\varphi$.
\item $E(\alpha)$ denote either $E_{\alpha}$ or $E_{\alpha}^*$, where $\alpha\in \hom(Q,P)$ or $\hom(Q,P^*)$.
\item $h$ denote the hyperbolic plane $H(R)$ and $h^{m} = h\perp h \perp \ldots \perp h = H(R)^{m}$.
\item  $\EO_R(q, h^m) = \EO_{R}(Q, H(R)^{m})$ and  $\og_R(q\perp h^m) = \og_{R}(Q\perp H(R)^{m})$.
\end{itemize}

\section{Normality of $\EO_R(q, h^m)$ in $\og_R(q\perp h^m)$}
In this section, we prove the normality of the elementary orthogonal group $\EO_R(q, h^m)$ of $\og_R(q\perp h^m)$, where $R$ is a commutative ring in which $2$ is invertible, $(Q,q)$ a quadratic space, and $h$ denote the hyperbolic plane $H(R)$.

Now we recall a decomposition result of R.A. Rao from \cite{MR727375}. 

\begin{lemma}[\cite{MR727375}*{Lemma~2.2}]\label{Rao84}
 Let $(Q,q)$ be a diagonizable quadratic $R$-space. Let $m\geq \dim \max R +1$. Then 
 \[\og_R(q\perp h^m) = \EO_R(q, h^m)\cdot\og_R(h^m) = \og_R(h^m)\cdot \EO_R(q,h^m)),\]where $\EO_R(q, h^m)\cdot\og_R(h^m)$ denote the subset 
 $\{\sigma_1\sigma_2 \,|\,\sigma_1\in \EO_R(q,h^m), \sigma_2\in \og_R(h^m)\}$ of $\og_R(q\perp h^m)$.
 \end{lemma}
 
 In \cite{MR610478} on page 290, R. Parimala remarked that if $\alpha \in O_k(h)$, where $k$ is a field of characteristic $\neq 2$, then there is a nonzero element $u \in R$ such that $\alpha = [u]\perp [u^{-1}]$ or  $[u]\top [u^{-1}].$ We prove that this also holds over a local ring.
\begin{lemma}\label{normal2}
Let $R$ be a local ring with maximal ideal $\mathfrak{m}$. Assume that 
$2R = R$. Let $\alpha \in O_R(h)$. Then there exists a unit $u \in R$ 
such that $\alpha = [u]\perp [u^{-1}]$ or $ [u]\top [u^{-1}].$
\end{lemma}
\begin{proof}
Let $\alpha = \begin{pmatrix} a & b\\ c & d\end{pmatrix} \in O_R(h)$ and $ \widetilde{\psi_1}$ denotes the matrix $\begin{pmatrix} 0 & 1\\1 & 0\end{pmatrix}.$ Then one has the relations $\alpha^t \widetilde{\psi_1} \alpha = \widetilde{\psi_1}$, $\alpha \widetilde{\psi_1}
= \widetilde{\psi_1} \alpha^{-1}$. From this one deduces that $\det(\alpha)^2 = 1$, 
and $2ac = 0 = 2bd$, $ad + bc = 1$, $\Delta c = -c$, $\Delta b = -b$, 
$\Delta a = a$, $\Delta d =  d$, where $\Delta = \det(\alpha)$. From these relations it is easy to deduce Parimala's result in the case 
when $R$ is a field. We shall use this below. 

\vspace{1.5mm}
Let the $\overline{\alpha}$ denote $\alpha$ modulo $\mathfrak{m}$. From above equations, it follows that 
$$\overline{\Delta} = \det(\overline{\alpha})   = \pm 1.$$ 
If $\overline{\Delta} = 1$, then $\overline{\alpha} = [u]\perp [u^{-1}]$, for some unit $u \in R/\mathfrak{m}$.
Hence, $a, d$ are units and $b, c \in \mathfrak{m}$. Since $ad = 1$ 
modulo $\mathfrak{m}$, let $ad = 1 + m$, for some $m \in \mathfrak{m}$. Now $1 = \Delta^2 = (1 + m -bc)^2 = 1 + (m-bc)(2 + m - bc)$, whence 
$m - bc = 0$. Hence $\Delta = 1 + m -bc = 1.$  But then the equations $\Delta b = -b$, and $\Delta c = -c$ give
$b = 0 = c$, and so $\alpha = [a]\perp [a^{-1}]$ as 
required. 

\vspace{1.5mm}
If $\overline{\Delta} = -1$, then $\overline{\alpha} = [u] \top [u^{-1}]$, for some unit $u \in R/\mathfrak{m}$. Consider $\beta = \alpha \widetilde{\psi_1} \in O_R(h)$. Then $\det(\beta) = 1$; whence $\beta = [b] \perp  [b^{-1}]$, for some unit $b \in R$. Hence $\alpha = [b]\top [b^{-1}]$, for some unit $b$ in $R$.
 \end{proof}
 \begin{lemma}\label{normal1}
  Let $R$ be a local ring with maximal ideal $\mathfrak{m}$. Assume that $2R=R$. Let $(Q,q)$ be a quadratic space over $R$ of rank $n$. Then the group $\og_R(h)$ normalises $\EO_R(q, h)$.
 \end{lemma}
\begin{proof}
The elements of $\og_R(h)$ are of the form \begin{eqnarray*}
\tau_u = [u]\perp [u^{-1}] & 
\mbox{ or }~ \sigma_u = [u]\top [u^{-1}].
\end{eqnarray*}The elementary generators of $\EO_R(q,h)$ are of the form $E_{\alpha_{1j}}$ or $E_{\beta_{1j}}^*$ for $1\leq j \leq \rank Q = n$. From the following set of equations
\begin{eqnarray*}
(\I_{n}\perp \tau_u )\ E_{\alpha_{1j}}\ (\I_{n}\perp\tau_u^{-1}) &= E_{u\alpha_{1j}},  \\ 
(\I_{n}\perp \tau_u )\ E_{\beta_{1j}}^* (\I_{n}\perp \tau_u^{-1}) &= E_{u\beta_{1j}}^*, \\
(\I_{n}\perp \sigma_1)\  E_{\alpha_{1j}} (\I_{n}\perp \sigma_1^{-1})& = E_{\alpha_{1j}},\\
(\I_{n}\perp \sigma_1)\  E_{\beta_{1j}}^* (\I_{n}\perp \sigma_1^{-1}) &= E_{\beta_{1j}}^*,
\end{eqnarray*}
it is clear that $\og_R(h)$ normalises $\EO_R(q, h)$.
\end{proof}
\begin{corollary}\label{normal4}
 Let $R$ be a local ring with maximal ideal $\mathfrak{m}$. Assume that $2R=R$. Let $(Q,q)$ be a quadratic space of rank $n$ over $R$. Then the elementary orthogonal group $\EO_R(q, h)$ is normal in $\og_R(q\perp h)$. 
\end{corollary}
\begin{proof}
By Lemma~\ref{Rao84}, we have $\og_R(q\perp h) = \EO_R(q, h)\cdot\og_R(h) = \og_R(h)\cdot \EO_R(q,h)$. By Lemma~\ref{normal1}, it follows that $\og_R(q\perp h) $ normalises $\EO_R(q, h)$.
\end{proof}
\begin{lemma}\label{normal3}
 Let $R$ be a local ring with maximal ideal $\mathfrak{m}$. Assume that $2R=R$. Let $(Q,q)$ be a quadratic space. Then the group $\og_R(h)$ normalises $\EO_R(q, h^m)$.
\end{lemma}
\begin{proof}
The elementary generators of $\EO_R(q,h^m)$ are either of the form $E_{\alpha_{ij}}$ or $E_{\beta_{ij}}^*$ for $\alpha,\beta \in \hom(Q,R^m)$, $1\leq i \leq m$, and $1\leq j\leq n$.

 For $1\leq i \leq m, 1\leq j\leq n$, we have the set of equations
 \begin{eqnarray*}
(\I_{n+2m-2}\perp \tau_u )\ E_{\alpha_{ij}}\ (\I_{n+2m-2}\perp\tau_u^{-1}) &= \begin{cases} 
E_{u\alpha_{ij}}, \mbox{ if } i=m\\
E_{\alpha_{ij}}, \mbox{ if } i\neq m.
\end{cases}\\ 
(\I_{n+2m-2}\perp \tau_u )\ E_{\beta_{ij}}^* (\I_{n+2m-2}\perp \tau_u^{-1}) &= \begin{cases} 
E_{u\beta_{ij}}^*, \mbox{ if } i=m\\
E_{\beta_{ij}}^*, \mbox{ if } i\neq m.
\end{cases} \\
(\I_{n+2m-2}\perp \sigma_1)\  E_{\alpha_{ij}} (\I_{n+2m-2}\perp \sigma_1^{-1})& = E_{\alpha_{ij}},\\
(\I_{n+2m-2}\perp \sigma_1)\  E_{\beta_{ij}}^* (\I_{n+2m-2}\perp \sigma_1^{-1}) &= E_{\beta_{ij}}^*.
\end{eqnarray*}
From these equations it is clear that $\og_R(h)$ normalises $\EO_R(q, h^m)$.
\end{proof}
\begin{lemma}\label{normal5}
 Let $(Q,q)$ be a quadratic space of rank $n$ over $R$. Then the orthogonal group $\og_R(q)$ normalises $\EO_R(q,h^m)$ for $m\ge 2$.
\end{lemma}
\begin{proof}
 Let $\alpha \in \og_R(q)$. Consider the elementary generator $\E_{\beta} \in \EO_R(q\perp h^m)$, where $\beta:Q\rightarrow R^m$. Then
 \begin{eqnarray*}
  (\alpha\perp \I_{2m})\ E_{\beta}\  (\alpha\perp \I_{2m})^{-1} &= \begin{pmatrix}
                                                            \I_{n}&0&-\alpha\beta^*\\
                                                            \beta\alpha^{-1}&I_{m}&
                                                            -\frac{\beta\beta^*}{2}\\
                                                            0&0&I_{m}
                                                           \end{pmatrix}.                                                           
\end{eqnarray*}
Note that $d_{B_q}^{-1} \circ (\alpha^T)^{-1} = \alpha \circ d_{B_q}^{-1}$. Therefore we get $(\beta\circ \alpha^{-1})^* = d_{B_q}^{-1} \circ (\beta\circ\alpha^{-1})^T = d_{B_q}^{-1}\circ(\alpha^{-1})^T \circ \beta^T = \alpha \circ d_{B_q}^{-1} \circ \beta^T = \alpha \circ \beta^*.$ Hence \[(\alpha\perp \I_{2m})\ E_{\beta}\  (\alpha\perp \I_{2m})^{-1} = E_{\beta\circ\alpha^{-1}} \in \EO_R(q, h^m).\]
\end{proof}

\begin{lemma}\label{lemdelta2}
 Let $(Q,q)$ be a quadratic space of rank $n=2r$ over $R$ and let $\varphi = \widetilde{\psi_r}$. Then $\EO_R(q, h^m) = \EO_{n+2m}(R)$, where $\EO_R(q,h^m)$ denote DSER elementary orthogonal group and $\EO_{n+2m}(R)$ denote the usual elementary orthogonal group.
\end{lemma}
\begin{proof}
 Let $\varphi = \widetilde{\psi_r}$. Then $q\perp h^m$ has the form $\widetilde{\psi}_{r+m}$. Let $\sigma$ denote the permutation matrix associated to this form where $\sigma$ is given by $$\sigma(l) = \begin{cases}
			l+1     & \mbox{ if } l \mbox{ odd }\\
			l-1	& \mbox{ if } l \mbox{ even }                                                                                                                                                                     \end{cases},$$ for $1\leq l \leq n+2m$. 
The elementary generators for $\EO_{n+2m}(R)$ with respect to this form are \
\begin{equation}\label{eq1}
oe_{kl}(a) = \I_{n+2m} + ae_{kl}-ae_{\sigma(l),\sigma(k)} \mbox{ if } k \neq \sigma(l) \mbox{ and } k < l, 
\end{equation}
 for $1\leq k \neq l \leq n+2m$ and $a\in R$. The elementary generators for $\EO_R(q, h^m)$ are 
 \begin{align}
  E_{\alpha_{kl}} &= E_{n+2k-1,l}(a_{kl})E_{\sigma(l), n+2k}(-a_{kl}),\label{eq2}\\
  E_{\beta_{kl}}^* &= E_{n+2k,l}(b_{kl})E_{\sigma(l), n+2k-1}(-b_{kl}),\label{eq3}
 \end{align}
 for $1\leq k \leq m$, $1\leq l \leq n$, $\alpha = (a_{kl})$ and $\beta = (b_{kl})$ for $a_{kl},b_{kl} \in R$.
 
 \vspace{1.5mm}
 From Eq.\eqref{eq1}, Eq.\eqref{eq2} and Eq.\eqref{eq3}, it is clear that for $1\leq k \leq m \mbox{ and } 1\leq l \leq n,$
 \[E_{\alpha_{kl}} = oe_{\sigma(l),n+2k}(-a_{kl})\mbox{ and } E_{\beta_{kl}}^* = oe_{\sigma(l),n+2k-1}(-b_{kl}).\]
 
 Hence \[\EO_R(q, h^m) \subseteq \EO_{n+2m}(R).\]
 To prove the reverse inclusion, it is enough to show that for $1\leq k,l \leq n$ and $n+1 \leq k,l \leq n+2m$, $k\neq \sigma(l)$, $k<l$  and $a\in R$, $oe_{kl}(a) \in \EO_R(q, h^m)$.

 \noindent For $1\leq k,l \leq n$ and for $1\leq i\leq m$,
 \begin{align*}
  oe_{kl}(a) &= oe_{\sigma(l),\sigma(k)}(-a)\\
	     &= [oe_{\sigma(l),n+2i}(-a), oe_{n+2i,\sigma(k)}(1)]\\
	     &= [oe_{\sigma(l),n+2i}(-a), oe_{k,n+2i-1}(-1)]\\
	     &= [E_{\alpha_{il}(a)}, E_{\beta_{i,\sigma(k)}(1)}^*].
 \end{align*}
 For $k=n+2i-1, l=n+2j-1, 1\leq i\neq j \leq m$ and for $1\leq s \leq n$,
 \begin{align*}
  oe_{kl}(a) &= [oe_{n+2i-1,s}(a), oe_{s,n+2j-1)}(1)]\\
	     &= [oe_{\sigma(s),n+2i}(-a), oe_{s, n+2j-1}(1)]\\
	     &= [E_{\alpha_{is}(a)}, E_{\beta_{j,\sigma(s)}(-1)}^*].
 \end{align*}
 For $k=n+2i-1, l=n+2j, 1\leq i,j \leq m$ and for $1\leq s \leq n$,
 \begin{align*}
  oe_{kl}(a) &= [oe_{n+2i-1,s}(a), oe_{s,n+2j)}(1)]\\
	     &= [oe_{\sigma(s),n+2i}(-a), oe_{s, n+2j}(1)]\\
	     &= [E_{\alpha_{is}(a)}, E_{\alpha_{j,\sigma(s)}(-1)}].
 \end{align*}
 For $k=n+2i, l=n+2j-1, 1\leq i < j \leq m$ and for $1\leq s \leq n$,
 \begin{align*}
  oe_{kl}(a) &= [oe_{n+2i,s}(a), oe_{s,n+2j-1)}(1)]\\
	     &= [oe_{\sigma(s),n+2i-1}(-a), oe_{s, n+2j-1}(1)]\\
	     &= [E_{\beta_{is}(a)}^*, E_{\beta_{j,\sigma(s)}(-1)}^*].
 \end{align*}
 For $k=n+2i, l=n+2j, 1\leq i\neq j \leq m$ and for $1\leq s \leq n$,
 \begin{align*}
  oe_{kl}(a) &= [oe_{n+2i,s}(a), oe_{s,n+2j)}(1)]\\
	     &= [oe_{\sigma(s),n+2i-1}(-a), oe_{s, n+2j}(1)]\\
	     &= [E_{\beta_{is}(a)}^*, E_{\alpha_{j,\sigma(s)}(-1)}].
 \end{align*}
 Thus we get  \[\EO_R(q, h^m) = \EO_{n+2m}(R).\]
 \end{proof}
 \begin{proposition}[\cite{aa1}*{Proposition~4.5}]\label{normal}
$ \og_R(q \perp h^{m-1})$ normalizes $ \EO_R(q,h^{m})$, where $q$ and $h^m$ are quadratic spaces over a commutative ring $R$ for $m\ge 2$.
\end{proposition}
To be self-contained, we sketch a different proof of this proposition.
\begin{lemma}\label{N1}
Let $(Q,q)$ be a diagonalizable quadratic $R$-space. The elementary orthogonal group $\EO_{4}(R)$ normalizes the DSER elementary orthogonal group $\EO_{R}(q,h^{m})$ for $m \ge 2$.
\end{lemma}
\begin{proof}
The elementary orthogonal group $\EO_{4}(R)$ has elementary generators of the form
$oe_{13}(a)$, $oe_{14}(b)$, $oe_{23}(c)$, $oe_{24}(d)$, where $oe_{kl}(x) = \I_{n+2m} + xe_{kl}-xe_{\sigma(l),\sigma(k)}$ if $k\neq \sigma(l)$  and $ k < l, $
 for $1\leq k \neq l \leq n+2m$ and $x\in R$. 
 \vspace{1.5mm}
 The elementary generators of $\EO_{R}(q,h^{m})$ are of the (matrix) form $$E_{\alpha_{ij}(x)} = \I_{n+2m} + x e_{n+2i-1,j} - d_{j}xe_{j,n+2i} - \frac{1}{2} d_{j} x^{2} e_{n+2i-1,n+2i},$$ 
 $$E_{\beta_{ij}(y)}^{*} = \I_{n+2m} + y e_{n+2i,j} - d_{j}ye_{j,n+2i-1} - \frac{1}{2} d_{j} y^{2} e_{n+2i,n+2i-1},$$
where $n = \rank (Q)$, $x,y \in R$, $d_{j} \in R$ is the $(j,j)$th entry of the diagonal matrix $B_{q}^{-1}$, $1\le i \le m$, $1\le j \le n$; $\alpha_{ij}(x)$ is an $m\times n$ matrix with $(i,j)$-th entry is $x$ and all other entries $0$ and $\beta_{ij}(y)$ is an $m\times n$ matrix with $(i,j)$-th entry is $y$ and all other entries $0$. 

\vspace{1.5mm}
We need to prove that each $oe_{kl}$ normalizes both $E_{\alpha_{ij}}$ and $E_{\beta_{ij}}^{*}$ for all $1\le i \le m$ and $1 \le j \le n$. We illustrate one example here:

\vspace{1.5mm}

Let $n=2$ and $m=2$. Then the elementary orthogonal group $\EO_{R}(q,h^{2})$ has elementary generators $E_{\alpha_{11}(a_{11})}$,  $E_{\alpha_{12}(a_{12})}$,  $E_{\alpha_{21}(a_{21})}$,  $E_{\alpha_{22}(a_{22})}$, $E_{\beta_{11}(b_{11})}^{*}$, $E_{\beta_{12}(b_{12})}^{*}$, $E_{\beta_{21}(b_{21})}^{*}$, $E_{\beta_{22}(b_{22})}^{*}$, where $a_{11}$, $a_{12}$, $a_{21}$, $a_{22}$, $b_{11}$,   $b_{12}$, $b_{21}$, $b_{22} \in R$.

\begin{align*}
oe_{13}(a)E_{\alpha_{11}(a_{11})}oe_{13}(-a) &= E_{\alpha_{11}(a_{11})},\\
oe_{13}(a)E_{\alpha_{12}(a_{12})}oe_{13}(-a) &= E_{\alpha_{12}(a_{12})},\\
oe_{13}(a)E_{\alpha_{21}(a_{21})}oe_{13}(-a) &= E_{\alpha_{11}(\frac{1}{2}aa_{21})}E_{\alpha_{21}(a_{21})}E_{\alpha_{11}(\frac{1}{2}aa_{21})} = E_{\alpha_{11}(aa_{21})+\alpha_{21}(a_{21})}, \\
oe_{13}(a)E_{\alpha_{22}(a_{22})}oe_{13}(-a) &= E_{\alpha_{12}(\frac{1}{2}aa_{22})}E_{\alpha_{22}(a_{22})}E_{\alpha_{12}(\frac{1}{2}aa_{22})} = E_{\alpha_{12}(aa_{22})+\alpha_{22}(a_{22})}, \\
oe_{14}(b)E_{\alpha_{11}(a_{11})}oe_{14}(-b) &= E_{\alpha_{11}(a_{11})},\\
 oe_{14}(b)E_{\alpha_{12}(a_{12})}oe_{14}(-b) &= E_{\alpha_{12}(a_{12})},\\
oe_{14}(b)E_{\alpha_{21}(a_{21})}oe_{14}(-b) &= E_{\alpha_{21}(a_{21})}, \\
oe_{14}(b)E_{\alpha_{22}(a_{22})}oe_{14}(-b) &= E_{\alpha_{22}(a_{22})},\\
oe_{23}(c)E_{\alpha_{11}(a_{11})}oe_{23}(-c) &= \left[ E_{\beta_{21}(\frac{1}{2}ca_{11})}^{*}, E_{\alpha_{11}(a_{11})}\right] E_{\beta_{21}(-ca_{11})}^{*}E_{\alpha_{11}(a_{11})},&\\
 oe_{23}(c)E_{\alpha_{12}(a_{12})}oe_{23}(-c) &=  \left[ E_{\beta_{22}(\frac{1}{2}ca_{12})}^{*}, E_{\alpha_{12}(a_{12})}\right] E_{\beta_{22}(-ca_{12})}^{*}E_{\alpha_{12}(a_{12})},&\\
 oe_{23}(c)E_{\alpha_{21}(a_{21})}oe_{23}(-c) &= \left[ E_{\beta_{11}(\frac{1}{2}ca_{21})}^{*}, E_{\alpha_{21}(a_{21})}\right] E_{\beta_{11}(ca_{21})}^{*}E_{\alpha_{21}(a_{21})},&\\
  oe_{23}(c)E_{\alpha_{22}(a_{22})}oe_{23}(-c) &=  \left[ E_{\alpha_{22}(a_{22})}, E_{\beta_{12}(\frac{1}{2}ca_{22})}^{*}\right] E_{\beta_{12}(ca_{22})}^{*}E_{\alpha_{22}(a_{22})},&\\
  oe_{24}(d)E_{\alpha_{11}(a_{11})}oe_{24}(-d) &= \left[ E_{\alpha_{21}(\frac{1}{2}da_{11})}, E_{\alpha_{11}(a_{11})}\right] E_{\alpha_{21}(-da_{11})} E_{\alpha_{11}(a_{11})},&\\
 oe_{24}(d)E_{\alpha_{12}(a_{12})}oe_{24}(-d) &=\left[ E_{\alpha_{22}(\frac{1}{2}da_{12})}, E_{\alpha_{12}(a_{12})}\right] E_{\alpha_{22}(-da_{12})} E_{\alpha_{12}(a_{12})},&\\
  oe_{24}(d)E_{\alpha_{21}(a_{21})}oe_{24}(-d) &= E_{\alpha_{21}(a_{21})},\\
  oe_{24}(d)E_{\alpha_{22}(a_{22})}oe_{24}(-d) &= E_{\alpha_{22}(a_{22})}.
\end{align*}
We will get similar relations for $E_{\beta_{11}(b_{11})}^{*}$, $E_{\beta_{12}(b_{12})}^{*}$, $E_{\beta_{21}(b_{21})}^{*}$, and $E_{\beta_{22}(b_{22})}^{*}$.
Hence it follows that the group $\EO_{4}(R)$ normalizes the DSER elementary orthogonal group $\EO_{R}(q,h^{2})$. 
\end{proof}
\begin{lemma}\label{N2}
Let $R$ be a local ring. The orthogonal group $\og_{R}(h^{m-1})$ normalizes the DSER elementary orthogonal group $\EO_{R}(q,h^{m})$ for $m\ge 2$.
\end{lemma}
\begin{proof}
For $m=2$, this lemma reduces to Lemma~\ref{normal1}. For $m > 2$, we have $\og_{R}(h^{m-1}) = \EO_{R}(h^{m-2},h)\cdot \og_{R}(h)$ by Lemma~\ref{Rao84}. The orthogonal group $\og_{R}(h)$ normalizes $\EO_{R}(q,h^{m})$ by Lemma~\ref{normal1}. The elementary orthogonal group $\EO_{R}(h^{m-2},h) \cong \EO_{2m-2}(R)$. Any elementary generator in $\EO_{2m-2}(R)$ can be considered as an element in some $\EO_{4}(R)$. Hence by Lemma~\ref{N1}, $\EO_{2m-2}(R)$ normalizes $\EO_{R}(q,h^{m})$.
\end{proof}
\begin{proof}[of Proposition~\ref{normal}]
Let $\alpha \in \og_R(q \perp h^{m-1})$. It is enough to prove that $\alpha$ normalises the elementary generators $E_{\beta}, E_{\gamma}^*$ in $\EO_R(q, h^m)$ for $\beta, \gamma \in \hom_{R}(Q,R^{m})$. Let $\theta(T) = \alpha E_{T\beta}\alpha^{-1} \in \EO_{R[T]}(q\otimes R[T], h^m \otimes R[T])$. Then $\theta(0) = \I$ and $(\theta(T))_{\mathfrak{m}} = \alpha_{\mathfrak{m}} {(E_{T\beta})}_{\mathfrak{m}}{(\alpha^{-1})}_{\mathfrak{m}}$.

\vspace{1.5mm}
If $(\theta(T))_{\mathfrak{m}} \in \EO_{(R[T])_{\mathfrak{m}}}(q\otimes {(R[T])_{\mathfrak{m}}},h^m \otimes {(R[T])_{\mathfrak{m}}})$ for all ${\mathfrak{m}} \in \max(R)$, then $\theta(T) \in \EO_{R[T]}(q\otimes R[T],h^m \otimes R[T])$ by Theorem~\ref{lg}. Hence we have $\theta(1) = \alpha E_{\beta}\alpha^{-1} \in \EO_R(q,h^m)$. So we can assume $R$ to be a local ring.
\vspace{1.5mm}
Assume that $R$ is a local ring. Then by Lemma~\ref{Rao84}, we have $ \og_R(q \perp h^{m-1}) = \EO_{R}(q,h^{m-1})\cdot \og_{R}(h^{m-1})$. By Lemma~\ref{N2}, we have $\og_{R}(h^{m-1})$ normalises $\EO_{R}(q,h^{m})$ and $\EO_{R}(q,h^{m-1}) \subset \EO_{R}(q,h^{m})$. Hence $ \og_R(q \perp h^{m-1})$ normalizes $ \EO_R(q,h^{m})$.
\end{proof}

Now we prove the normality of the DSER group $\EO_{R}(Q, H(P))$ when $P$ is free $R$-module.
\begin{theorem}\label{main1}
  $\EO_{R}(q, h^m) \vartriangleleft \og_{R}(q\perp h^m)$, where $q$ and $h^m$ are quadratic spaces over a commutative ring $R$ and $m\ge 2$.
\end{theorem}
\begin{proof}
Let $\alpha \in \og_R(q \perp h^m)$ and $E_{\beta}, E_{\gamma}^*$, elementary generators in $\EO_R(q, h^m)$. Let $$\theta(T) = \alpha E_{T\beta}\alpha^{-1} \in \EO_{R[T]}(q\otimes R[T], h^m \otimes R[T]).$$ Then $\theta(0) = \I$ and $(\theta(T))_{\mathfrak{m}} = \alpha_{\mathfrak{m}} {(E_{T\beta})}_{\mathfrak{m}}{(\alpha^{-1})}_{\mathfrak{m}}$. If $(\theta(T))_{\mathfrak{m}} \in \EO_{(R[T])_{\mathfrak{m}}}(q\otimes {(R[T])_{\mathfrak{m}}},h^m \otimes {(R[T])_{\mathfrak{m}}})$ for all ${\mathfrak{m}} \in \max(R)$, then by Theorem~\ref{lg}, $$\theta(T) \in \EO_{R[T]}(q\otimes R[T],h^m \otimes R[T]).$$ Hence we have $\theta(1) = \alpha E_{\beta}\alpha^{-1} \in \EO_R(q,h^m)$. So we can assume $R$ to be a local ring.

\vspace{1.5mm}
Assume that $R$ is a local ring. Then $q$ is diagonalizable. By Lemma~\ref{Rao84}, we have the decomposition
\[\og_R(q \perp h^m) = \EO_R(q, h^m)\cdot \og_R(h^m) = \EO_R(q,h^{m})\cdot\EO_R(h^{m-1},h)\cdot\og_R(h).\]
Note that by Lemma~\ref{normal3}, $\og_R(h)$ normalises $\EO_R(q, h^m)$. By Lemma~\ref{lemdelta2}, we conclude that $\EO_R(h^{m-1}, h) = \EO_{2m}(R)$. Since any elementary generator of $\EO_{2m}(R)$ is an element of some $\EO_{{2m-2}}(R)$ for $m>2$, it will be contained in $\bigcup \og_R(h^{m-1}) \subset \og_R(q\perp h^{m-1})$. But by Proposition~\ref{normal}, we have  $\og_R(q,h^{m-1})$ normalises $\EO_R(q,h^{m})$. Hence the group $\og_R(q\perp h^m)$ normalises $\EO_R(q, h^m)$.
\end{proof}

\section{Local-Global principle for the DSER group}
In this section, we prove  an analogue of Quillen's Local-Global principle(extended case) for the DSER elementary orthogonal group. 
We take $M = Q\perp H(P) $ and $M[X]$ will denote $(Q\perp H(P) )[X]$. We assume that the rank of the projective module $Q$ is $\geq 1$ and $\rank P \geq 2$. 
\vspace{1.5mm}

We recall lemmas from \cite{aarr}. 
\begin{lemma}[\cite{aarr}*{Lemma~5.5}]\label{split} 
Let $Q$, $P$ be free $R$-modules of rank $n$ and $m$ respectively. The kernel of the evaluation map $\EO_{A[X]}(Q[X],H(P[X]))\rightarrow \EO_A(Q,H(P))$ at $X=0$ is generated by the elements of the type  $\gamma E\left(X\alpha_{ij}(X)\right) \gamma^{-1}$, where $\gamma \in\EO_A(Q,H(P))$, $\alpha_{ij}(X) \in \hom(Q[X],P[X])$ or $\hom(Q[X],P^{\ast}[X])$.
\end{lemma}
\begin{lemma}[(Dilation Lemma)\cite{aarr}*{Lemma~5.7}]\label{dilation}
Let $Q,P$ be free modules of rank $n$ and $m$
respectively. Let $s$ be a non-nilpotent element of $R$ and $M = Q
\!\perp\! H(P).$ Let $\theta(X) \in~\og_{R[X]}(M[X])$ with $\theta(0)=\I.$
Let $Y,Z \in \hom(Q,P)$ or $\hom(Q,P^*)$. If
$\theta_s(X) :=$ $(\theta(X))_s$ is  an element of $\EO_{{R_s[X]}}\left(Q \otimes R_{s}[X], H(R)^{m} \otimes R_{s}[X]\right)$, then
for $N \gg 0$ and for all $b \in (s^N) R$, we have $\theta(bX) \in
\EO_{R[X]}\left(Q \otimes R[X], H(R)^{m} \otimes R[X]\right)$.
\end{lemma}
\begin{theorem}[(Local-Global Principle)\cite{aarr}*{Theorem~5.3}]\label{lg}
 Let $\theta(X) \in \og_{R[X]}(M[X])$ with $\theta(0) = \I$. If for all maximal ideals $\mathfrak{m}$ of $R$, $\theta(X)_\mathfrak{m} \in \EO_{R_{\mathfrak{m}}[X]}(Q \otimes R_\mathfrak{m}[X], H(R)^{m}\otimes R_\mathfrak{m}[X] )$, then $\theta(X) \in \EO_{R[X]}(Q \otimes R[X], H(R)^{m} \otimes R[X])$.
\end{theorem}

We state here some useful facts.

\begin{enumerate}\label{hom1}
\item Let $R$ be a ring and $S$ be a multiplicative subset of $R$. Let $M$ be a finitely presented $R$-module and $N$ be any $R$-module. Then we have a natural isomorphism,
\[\eta : S^{-1} (\hom_R(M,N)) \rightarrow   \hom_{S^{-1}R}(S^{-1}M,S^{-1}N).\]
\item \label{hom2}
Let $R$ be a ring and $M$ be a finitely presented $R$-module and let $N$ be any $R$ module. Then we have a natural isomorphism,
\[\gamma: \hom_R(M,N)[X] \rightarrow \hom_{R[X]}(M[X],N[X]).\]
\end{enumerate}

We state a variation of \cite{aarr}*{Lemma~5.6} which can be proved by the similar arguments of \cite{aarr}*{Lemma~5.6}. we also refer the reader to \cite{aa} for the various commutator relations involved in proving this lemma.

\begin{lemma}\label{gen3} 
Let $Q$ be a quadratic $R$-space and $P$ be a finitely generated projective $R$-module.
Let $s$ be a non-nilpotent element of $R$. Fix $r\in \mathbb{N}$. Given an integer $N\in \mathbb{N}$ and $x\in R$, then for integers $N_{t}\geq N$ and $x_{t} \in R$, there exists a product decomposition 
\[
  E\left(W_{ij}\right)E\left(s^NxY_{kl}\right)E\left(-W_{ij}\right)
   = \prod_{t=1}^\mu E\left(s^{N_t}x_tZ_{p_tq_t}\right),
\]where $W,Y,Z \in \hom_{R_{s}}(Q_{s},P_{s})$ or $W,Y,Z \in \hom_{R_{s}}(Q_{s},P_{s}^{*})$, $i,k,p_t \in \{1,2,...,m\}$ and $j,l,q_t \in \{1,2,...,n\}$ for every integer $1\leq t \leq \mu$.
\end{lemma}
\begin{proof}
Use the similar arguments as in \cite{aarr}*{Lemma~5.6} which involves various commutator relations among the elementary generators of $\EO_{R_{s}}(Q_{s},P_{s})$.
\end{proof}
\begin{corollary}\label{gen2}
Let $Q$ be a quadratic $R$-space and $P$ be a finitely generated projective $R$-module. If $\varepsilon = \varepsilon_1\varepsilon_2\ldots\varepsilon_k$, where each $\varepsilon_j$ is an elementary generator of the type $E_{{\alpha_{ij}}}$ or $E_{\beta_{ij}}^{*}$ with $\alpha \in \hom_{R_{s}}(Q_{s},P_{s})$ or $\beta \in \hom_{R_{s}}(Q_{s},P_{s}^*)$.  For ${\mathnormal{l}}>0$, $1\leq r \leq m$, $1\leq p \leq n$ and $x\in R$, there is a product decomposition
\[\varepsilon E(s^{2^{\mathnormal{k}}\mathnormal{l}}xW_{rp})\varepsilon^{-1} = \prod_{t=1}^{\mu_{k}} E\left(s^{\mathnormal{l}}x_{t}Y_{r_tp_t}\right),\]where $W,Y \in \hom(Q_{s},P_{s})$ or $\hom(Q_{s},P_{s}^*)$, $1\leq r_{t} \leq m$, $1\leq p_{t}\leq n$ and $x_{t}\in R$ chosen suitably.

\end{corollary}
\begin{proof}
Apply Lemma~\ref{gen3} repeatedly.
\end{proof}

\begin{proposition}[(Dilation Principle)]
Let $(Q,q)$ be a quadratic $R$-space of rank $n\ge 1$ and let $P$ be a finitely generated projective $R$-module of rank $m \ge 2$. Let $M= Q\perp H(P) $. Let $s$ be a non-nilpotent element in $R$ such that $P_s$ and $Q_s$ are free. Let $\theta(X) \in \og_{R[X]}(M[X])$ with $\theta(0) = \I$. Suppose that $\theta_s(X) \in \EO_{R_s[X]}(Q_s[X],R_s[X]^{m})$. Then there exists $\widehat{\theta}(X) \in \EO_{R[X]}(Q[X], P[X])$ and $\ell\gg0$ such that $\widehat{\theta}(X)$ localizes to $\theta(bX)$ for some $b\in (s^{\ell})$ and $\widehat{\theta}(0) = \I$.
\end{proposition}
\begin{proof}
     Since $\theta(0) =\I$, we can write $\theta_s(X) = \textstyle{\prod_{t=1} ^{\nu}\gamma_{t} E\left(X{h}_{t}(X)W_{i_{t}j_{t}}\right) {\gamma_{t}}^{-1}}$, where ${{h}_{t}}(X) \in R_s[X]$ and $\gamma_{t} \in \EO_{R_s}(Q_s, P_{s})$. For ${\mathnormal{l}}>0$,  we can write $\theta_{s}(s^{\mathnormal{l}}X)$ as 
     $$\theta_s(s^{\mathnormal{l}}X) = \textstyle{\prod_{k=1}^{\nu} \gamma_{k} E\left(s^{\mathnormal{l}}X{h_{k}}(s^{\mathnormal{l}}X)W_{i_{k}j_{k}}\right) {\gamma_{k}}^{-1}}.$$ 
 Let $\gamma_{k} = \varepsilon_1 \varepsilon_2 \dots \varepsilon_{r_{k}}$ and $\mathnormal{l}=2^{{r_{k}}-1}$. Using Corollary~\ref{gen2}, we get 
 $$\theta_s(s^{2\mathnormal{l}}X) = \prod_{k=1}^{\nu} \prod_{t=1}^{\nu_{r_{k}}}E\left(s^{\mathnormal{l}}g_{t}(X)W_{i_{t}j_{t}}\right), \mbox{ where }W_{i_{t}j_{t}} \in \hom(Q_s,P_s).$$  
 
Since $Q_s$ and $P_s$ are free $R_s$-modules, we have
\[ P_s[X,Z]  \cong R_s^{2m}[X,Z]\cong P_s[X,Z]^{*} \mbox{ and }Q_s[X,Z] \cong R_s^{n}[X,Z].\] Using this isomorphism, the polynomials in $P_s[X,Z] $ can be regarded as the bilinear form as follows: For  $x=(z_1,z_2,\ldots, z_n, x_1,x_2,\ldots, x_{2m})$, $y=(w_1,w_2,\ldots, w_n, y_1,y_2,\ldots, y_{2m}) \in R_s^{n+2m}[X,Z]$; 
\[\langle x,y \rangle = (z_1,z_2,\ldots, z_n)\varphi(w_1,w_2,\ldots, w_n)^T\perp(x_1,x_2,\ldots, x_{2m}){\widetilde{\psi_m}}(y_1,y_2,\ldots, y_{2m})^T,\]where $\varphi$ is a diagonal invertible matrix and ${\widetilde{\psi_m}}$ denotes the standard hyperbolic form $${\sum_{i=1}^{2m}e_{2i-1,i}+\sum_{i=1}^{2m}e_{2i,2i-1}.}$$

Let $e_i^*$ be the standard basis for $R_s^{2m}$. Consider the element $Zg_t(X)e_i^*\in R_s^{2m}[X,Z]$ as an element in $(P_s)[X,Z]$.  By Corollary~\ref{gen2}, we can choose $N_t > 0$ such that $N_{t}$ is the maximum power of $s$ occurring in the denominator of $Zg_t(X)e_i^*$. Now using fact~\ref{hom2}, we can consider the element $Zg_t(X)e_i^*$ as a polynomial in $Z$. Applying the homomorphism $Z\mapsto s^{N_t}Z$, it follows that $\theta(bXZ^{2\mathnormal{l}})$ is defined over $R[X]$. That is, we have 
\begin{equation*} \theta(bXZ^{2\mathnormal{l}}) =
\textstyle{\prod_{k=1}^{\nu} \prod_{t=1}^{\nu_{r_{k}}}
E\left(s^{N_t}Zg_t(X)W_{i_tj_t}\right)} \in \EO_{R[X]}\left(Q[X], P[X]\right) \mbox{ for all } b\in(s^{\ell}) R, (\ell\gg0).
\end{equation*} 

Now substitute $Z=1$, then by Lemma~\ref{dilation}, there exists $\ell \gg0$ such that $\widehat{\theta}(X)\in \EO_{R[X]}(Q[X], P[X]))$ localizes to $\theta(bX)$ for some $b\in (s^{\ell})$ and $\widehat{\theta}(0) = \I$.

\end{proof}

\begin{theorem}[(Local-Global Principle)]\label{LGP}
Let $R$ be a commutative ring with identity in which $2$ is invertible. Let $(Q,q)$ be a quadratic $R$-space of rank $n\ge 1$ and let $P$ be a projective $R$-module of rank $m\geq 2$. Let $M= Q \perp H(P)$. Assume that for every maximal ideal $\mathfrak{m}$ of $R$, the module $M_{\mathfrak{m}}$ is isomorphic to $R_{\mathfrak{m}}^{n+2m}$ with the bilinear form $\varphi \perp {\overset{\sim}{\psi}}_{m}$. 

Let $\theta(X) \in \og_{R[X]}(M[X])$ with $\theta(0) = \I$ and $\theta_{\mathfrak{m}}(X) \in \EO_{R_{\mathfrak{m}}[X]}(Q_{\mathfrak{m}}[X],R_{\mathfrak{m}}[X]^{m})$ for all ${\mathfrak{m}}\in \max(R)$. Then $\theta(X) \in \EO_R(Q,H(P))$.
\end{theorem}

\begin{proof}
Use the Dilation Principle (Lemma~\ref{dilation}) and D. Quillen's arguments in \cite{MR0427303} to deduce the Local-Global Principle from the Dilation Principle.
\end{proof}
\section{Normality of $\EO_{R}(Q, H(P))$ in $\og_{R}(Q \perp H(P))$}
In this section, we prove the main result of this paper. Let $P$ be  a finitely generated projective module.
\begin{theorem}\label{main2}
  $\EO_{R}(Q, H(P))$ is a normal subgroup of  $\og_{R}(Q \perp H(P))$, where $Q$ and $H(P)$ are quadratic spaces over a commutative ring $R$, where $\rank(Q) \ge 1$ and $\rank(P) \ge 2$.
\end{theorem}

\begin{proof}
Let $\eta \in \og_R(Q \perp H(P))$ and $E_{\beta}, E_{\gamma}^*$, elementary generators in $\EO_R(Q, H(P))$. For any maximal ideal $\mathfrak{m}$ of $R$, the $R_{{\mathfrak{m}}}$-modules $Q_{\mathfrak{m}}$and $H(P)_{\mathfrak{m}}$ are free and $H(P)_{\mathfrak{m}} \cong h^{m}$.

\vspace{1.5mm}
Let $\theta(T) = \eta E_{T\beta}\eta^{-1} \in \EO_{R[T]}(Q\otimes R[T], H(P) \otimes R[T])$. Then $\theta(0) = \I$ and 
$$(\theta(T))_{\mathfrak{m}} = \eta_{\mathfrak{m}} {(E_{T\beta})}_{\mathfrak{m}}{(\eta^{-1})}_{\mathfrak{m}} \in \EO_{{R_{\mathfrak{m}}[T]}}(Q\otimes {R_{\mathfrak{m}}[T]}, H(P) \otimes {{R_{\mathfrak{m}}[T]}})$$
 for all ${\mathfrak{m}} \in \max(R)$. By Theorem~\ref{LGP}, $\theta(T) \in \EO_{R[T]}(Q\otimes R[T], H(P) \otimes R[T])$ . Therefore, we have $\theta(1) = \eta E_{\beta}\eta^{-1} \in \EO_R(Q,H(P))$. \end{proof}

\begin{acknowledgements}\label{ackref}
The first author thanks the School of Mathematics, Tata Institute of Fundamental Research, Mumbai for its hospitality during this work. 
\end{acknowledgements}

\end{document}